\documentclass[11pt]{amsart}

\usepackage{amsmath,amssymb,amsthm,amsfonts,mathrsfs}
\usepackage{fullpage}
\usepackage[shortlabels]{enumitem}
\usepackage[colorlinks=true,
linkcolor=blue,
anchorcolor=blue,
citecolor=red
]{hyperref}
\allowdisplaybreaks 
\newtheorem{theorem}{Theorem}[section]
\newtheorem{lemma}[theorem]{Lemma}

\newtheorem{proposition}[theorem]{Proposition}

\theoremstyle{definition}

\theoremstyle{remark}
\newtheorem{remark}[theorem]{Remark}

\newcommand{\A}{\mathbb{A}}

\newcommand{\GL}{\mathrm{GL}}

\newcommand{\SL}{\mathrm{SL}}

\newcommand{\Sp}{\mathrm{Sp}}

\newcommand{\on}{\operatorname}
\newcommand{\ord}{\on{ord}}

\author[P. Yan]{Pan Yan}
\address{Department of Mathematics, The University of Arizona, Tucson, AZ 85721, USA}
\email{panyan@math.arizona.edu}

\makeatletter
\@namedef{subjclassname@2020}{\textup{}2020 Mathematics Subject Classification}
\makeatother
\date{\today}
\title{The unramified computation of a Shimura integral for $\mathrm{SL}(2)\times \mathrm{GL}(2)$}

\begin{document}

\begin{abstract}
In this note, we revisit the Rankin-Selberg integral of Shimura type for generic representations of $\mathrm{SL}_2\times \mathrm{GL}_2$, constructed by Ginzburg, Rallis, and Soudry. We give a different and more ``intrinsic'' proof of the unramified computation. In contrast to their proof we avoid local functional equation for the general linear groups but use the Casselman-Shalika formulas for unramified Whittaker functions for $\mathrm{SL}_2$ and $\mathrm{GL}_2$. 
\end{abstract}

\maketitle


\section{Introduction}

In \cite{GinzburgRallisSoudry1998}, Ginzburg, Rallis and Soudry constructed global integrals of Shimura type, which represent the (partial) tensor product $L$-function for a pair of irreducible automorphic cuspidal generic representations, one of $\mathrm{Sp}_{2n}$, and the other of $\mathrm{GL}_k$. They presented two different constructions that are ``dual'' to each other: one for the case $n\ge k$ and one for the case $n<k$. 
The integrals for the case $n=k$ were also constructed by Gelbart and Piatetski-Shapiro in a prior work \cite{GelbartPiatetski-Shapiro1987}. These integrals have been used to construct explicit functorial liftings from the general linear groups to the symplectic groups (see \cite{GinzburgRallisSoudry1999, GinzburgRallisSoudry2011}).

In this note, we consider the construction of Ginzburg, Rallis and Soudry \cite{GinzburgRallisSoudry1998} in the low rank case $n=1$, $k=2$.
To give more details about their global integral in this case, let us introduce some notations. Let $F$ be a global field with the ring of adeles $\mathbb{A}$. Let $\psi$ be a non-trivial additive character on $F\backslash \mathbb{A}$. We denote by $\omega_\psi$ the Weil representation of $\widetilde{\mathrm{SL}}_2(\mathbb{A})\ltimes\mathcal{H}(\mathbb{A})$, where $\mathcal{H}$ is the Heisenberg group of dimension 3, corresponding to the character $\psi$. Let $ \widetilde{\theta}_{\Phi}$ be a theta series associated to a Schwartz function $\Phi\in \mathcal{S}(\mathbb{A})$. Let $P=M_P\ltimes N_P$ and $Q=M_Q\ltimes N_Q$ be the Siegel parabolic subgroup and the Klingen parabolic subgroup of $\mathrm{Sp}_4$ respectively. 
The unipotent group $N_Q$ has a structure of the Heisenberg group $\mathcal{H}$; see \eqref{eq-NQ-Heisenberg}.
Let $\tau$ be an irreducible automorphic cuspidal representation of $\mathrm{GL}_2(\mathbb{A})$. 
Let $\widetilde{E}(g, \tilde{f}_{\tau, s})$ be a Siegel Eisenstein series on $\widetilde{\mathrm{Sp}}_4(\mathbb{A})$, associated to a smooth holomorphic $\widetilde{K}$-finite section 
\begin{equation*}
    \tilde{f}_{\tau, s}\in \mathrm{Ind}_{\widetilde{P}(\mathbb{A})}^{\widetilde{\mathrm{Sp}}_4(\mathbb{A})}(\tau \otimes   |\det|^s \otimes \gamma_\psi^{-1} ),
\end{equation*} 
where $\gamma_\psi$ is the inverse of the Weil factor attached to the character $\psi$ regarded as a function on $\GL_2(\A)$ via the pullback of $\gamma_\psi$ by the determinant map, and $\widetilde{K}$ is the preimage of the standard maximal compact subgroup $K\subset \mathrm{Sp}_4(\mathbb{A})$ in $\widetilde{\mathrm{Sp}}_4(\mathbb{A})$. 
Let $\pi$ be an irreducible automorphic cuspidal $\psi^{-1}$-generic representation of $\mathrm{SL}_2(\mathbb{A})$. For a cusp form $\varphi_\pi\in V_\pi$, we consider the following integral
\begin{equation}
I(\varphi_\pi,  \widetilde{\theta}_{\Phi},\widetilde{E}(\cdot, \tilde{f}_{\tau, s}))= \int\limits_{\mathrm{SL}_2(F)\backslash \mathrm{SL}_2(\mathbb{A})} \int\limits_{N_Q(F)\backslash N_Q(\mathbb{A})}  \varphi_\pi(g)  \widetilde{\theta}_{\Phi}(ug) \widetilde{E}(u t(g), \tilde{f}_{\tau, s})du dg.
\label{eq-introduction-global-integral}
\end{equation}
Here, for $g\in \text{SL}_2$, $t(g)=\mathrm{diag}(1, g, 1)$. We remark that the function $\widetilde{\theta}_{\Phi}(ug) \widetilde{E}(ut(g), f_{\tau, s})$ can be viewed as a function on $\mathrm{SL}_2(\mathbb{A})\ltimes N_Q(\mathbb{A})$, and \eqref{eq-introduction-global-integral} is the special case when $n=1, k=2$ of the integrals considered in \cite[Section 5]{GinzburgRallisSoudry1998}.

For $h\in \widetilde{\mathrm{Sp}}_4(\mathbb{A})$, we denote
\begin{equation*}
\begin{split}
\tilde{f}_{\mathcal{W}(\tau,\psi_2), s}(h)= \int\limits_{F\backslash \mathbb{A}}  \tilde{f}_{\tau, s}  \left(   \left(\begin{smallmatrix}  1 & x &&\\&1&&\\&&1&-x\\&&&1\end{smallmatrix}\right)  h \right) \psi_2(x)dx
\end{split}
\end{equation*}
where $\psi_2(x)=\psi(2x)$. The model $\mathcal{W}(\tau, \psi_2)$ is the Whittaker model of the representation $\tau$ with respect to the character $\left(\begin{smallmatrix} 1 & x\\ & 1\end{smallmatrix}\right) \mapsto \psi_2(x)=\psi(2x)$. The space of $\mathcal{W}(\tau, \psi_2)$ consists of smooth functions $W_{\phi, \psi_2}$ on $\mathrm{GL}_2(\mathbb{A})$ of the form 
$$
W_{\phi, \psi_2}(g)=\int_{N_2(F)\backslash N_2(\mathbb{A})}\phi(ng)\psi^{-1}_2(n)dn, \ \ \ \phi\in V_\tau, \ \ \ N_{2}=\left\{\left(\begin{smallmatrix} 1 & x\\ & 1\end{smallmatrix}\right) \right\},
$$
and the group $\mathrm{GL}_2(\mathbb{A})$ acts in the space  of $\mathcal{W}(\tau, \psi_2)$ by right translation.

The main results on the integral~\eqref{eq-introduction-global-integral} are summarized in the following theorem. 

\begin{theorem}[Ginzburg-Rallis-Soudry \cite{GinzburgRallisSoudry1998}]
\leavevmode
\begin{enumerate}[(i)]
    \item The integral $I(\varphi_\pi,  \widetilde{\theta}_{\Phi},\widetilde{E}(\cdot, \tilde{f}_{\tau, s}))$ is absolutely convergent when $\mathrm{Re}(s)\gg 0$ and can be meromorphically continued to all $s\in \mathbb{C}$. Moreover, when $\mathrm{Re}(s) \gg 0$, the integral $I(\varphi_\pi,  \widetilde{\theta}_{\Phi},\widetilde{E}(\cdot, \tilde{f}_{\tau, s})) $ unfolds to 
    \begin{equation}
    \int\limits_{N_2(\mathbb{A})\backslash \mathrm{SL}_2(\mathbb{A})}   W_{\varphi_\pi}(g)   \int\limits_{N_Q^0(\mathbb{A})} \omega_\psi(rg)\Phi(1) \tilde{f}_{\mathcal{W}(\tau,\psi_2), s}  \left( \gamma r t(g)  \right) dr dg,
    \label{eq-global-integral-prop-afterunfolding}
    \end{equation}
    where $W_{\varphi_\pi}\in \mathcal{W}(\pi, \psi^{-1})$ is the $\psi^{-1}$-Whittaker function of $\varphi_\pi$, and
    \begin{equation*}
    \gamma=\left(\begin{smallmatrix} & 1 &&\\ &&&-1\\1&&&\\&&1& \end{smallmatrix}\right), \, \, \, N_Q^0=\left\{ \left(\begin{smallmatrix} 1 & x &  & z\\ &1 & &\\ &&1&-x\\&&&1 \end{smallmatrix} \right)\right\}\subset N_Q.
    \end{equation*}
    \label{thm-Introduction-part(i)}\\
    
    \item At a finite local place $\nu$ when all data are unramified and normalized, the local integral is equal to 
\begin{equation*}
 \frac{L(\pi_\nu\times \tau_\nu, s+\frac{1}{2})}{L(\tau_\nu, \mathrm{Sym}^2, 2s+1)} .
\end{equation*}\label{thm-Introduction-part(ii)}
\end{enumerate}
\label{thm-Introduction}
\end{theorem}

\begin{remark}
We refer the reader to \cite[Theorem 5.1]{GinzburgRallisSoudry1998} for a proof of Theorem~\ref{thm-Introduction}~\ref{thm-Introduction-part(i)}. We point out that since we use a different isomorphism (see Remark~\ref{remark-1}) between the unipotent group $N_Q$ and the Heisenberg group $\mathcal{H}$ than that of \cite{GinzburgRallisSoudry1998}, after the unfolding process, we get the section $\tilde{f}_{\mathcal{W}(\tau,\psi_2), s}$ in \eqref{eq-global-integral-prop-afterunfolding} (while in \cite[Theorem 5.1]{GinzburgRallisSoudry1998} the global integral unfolds to an expression involving the section $\tilde{f}_{\mathcal{W}(\tau,\psi), s}$).
\end{remark}

The proof of Theorem~\ref{thm-Introduction}~\ref{thm-Introduction-part(ii)} in \cite{GinzburgRallisSoudry1998} is based on similar ideas in \cite{Soudry1993}, and it
involves several key ingredients. The first one is a formal identity, analogous to the one proved in \cite[Section 11.4]{Soudry1993}, which relates between the local integral and certain local gamma factor; see \cite[Proposition 6.1]{GinzburgRallisSoudry1998}. The second one is the application of the Casselman-Shalika formula, from which one obtains the normalizing factor for certain unramified Whittaker functional. The third is the unramified computation for the local integral for $\Sp_{2n}\times \GL_k$ in the case $n\ge k$. Moreover, the proof of Theorem~\ref{thm-Introduction}~\ref{thm-Introduction-part(ii)} in \cite{GinzburgRallisSoudry1998}  also involves the local functional equation for $\mathrm{GL}_n\times \mathrm{GL}_k$ developed in \cite{JacquetPiatetski-ShapiroShalika1983}. 

The goal of this paper is to give a direct and more ``intrinsic'' proof of Theorem~\ref{thm-Introduction}~\ref{thm-Introduction-part(ii)}, i.e., without resorting to the local functional equation for the general linear groups. 
Our method is based on the Casselman-Shalika formulas for $\SL_2$ and $\GL_2$; see Theorem~\ref{thm-local-unramified-main-result}. Henceforth until the end of the paper, we drop the reference to the local place $\nu$ to ease notation.

\begin{remark}
We remark that in \cite{Kaplan2010, Kaplan2012} Kaplan also obtained direct unramified computations for local integrals, but for orthogonal groups.
\end{remark}

\begin{remark}
We also remark that in the general case when $n<k$, there is an additional unipotent integration involving the section in the global integral in \cite{GinzburgRallisSoudry1998} after unfolding (i.e., the global integral unfolds to a triple integral as opposed to a double integral in \eqref{eq-global-integral-prop-afterunfolding}). However, when $n=k-1$ (such as the case $n=1, k=2$ we consider in this paper), this unipotent integration is absent, making the local unramified integral closer to the even rank case for which a direct unramified computation is known. Although we only consider the low rank case $n=1, k=2$ in this paper, we hope that our calculation will shed light on the techniques for direct unramified computation. 
\end{remark}

The organization of this paper is as follows. In Section~\ref{section-preliminaries},  we introduce general notations, the Weil representation, and the $L$-functions that are relevant in this paper. 
In Section~\ref{section-unramified}, we compute the local integral when all data are unramified, and prove Theorem~\ref{thm-Introduction}~\ref{thm-Introduction-part(ii)}.

\section{Preliminaries}
\label{section-preliminaries}
\subsection{Notations}
Let $F$ be a non-Archimedean local field with ring $\mathcal{O}_F$ of integers, with a fixed  uniformizer $\varpi$. 
We assume that the residue characteristic of $F$ is odd, and  let $q$ be the cardinality of the residue field. Let $\ord$ be the valuation function on $F$. The absolute value $|\cdot |$ on $F$ is normalized so that $|\varpi|=q^{-1}$. We fix a  non-trivial additive unramified character $\psi$ of $F$. For any $a\in F^\times$, the character $\psi_a$ is defined by $\psi_a(x)=\psi(ax)$. The local Hilbert symbol is denoted by $(\ , \ )_F$.

For a positive integer $n$, we let $J_n$ be the $n\times n$ matrix
whose antidiagonal entries are ones and all other entries are zeros. We realize the symplectic group $\Sp_{2n}(F)$ as
\begin{equation*}
\Sp_{2n}(F)=\left\{g \in \GL_{2 n}(F) : {}^t g \left(\begin{smallmatrix}
 & J_n\\
 -J_n &
\end{smallmatrix}\right) g=\left(\begin{smallmatrix}
 & J_n\\
 -J_n &
\end{smallmatrix}\right)\right\}.
\end{equation*} 
Let $\widetilde{\mathrm{Sp}}_{2n}(F)$ be the metaplectic double cover of $\mathrm{Sp}_{2n}(F)$, so we have an exact sequence
$$
1  \longrightarrow \{\pm 1\} \longrightarrow \widetilde{\mathrm{Sp}}_{2n}(F)   \longrightarrow \mathrm{Sp}_{2n}(F)   \longrightarrow 1.
$$
As a set, we may write  elements of $\widetilde{\mathrm{Sp}}_{2n}(F)$ as pairs $(g, \varepsilon)$ where $g\in\mathrm{Sp}_{2n}(F)$, $\varepsilon\in \{\pm 1\}$, with group law given by
\begin{equation*}
    (g_1, \varepsilon_1)\cdot (g_2, \varepsilon_2)=(g_1 g_2, \varepsilon_1 \varepsilon_2\cdot c(g_1, g_2) )
\end{equation*}
where $c$ is the Ranga Rao's $2$-cocycle on $\mathrm{Sp}_{2n}(F)$ valued in $\{\pm 1\}$ described in \cite{Rao1993}. 
We refer the reader to \cite{BerndtSchmidt1998} for more details.

In this paper, we will only consider the small rank symplectic groups when $n=1$ and $n=2$.
For $n=2$, the Siegel parabolic subgroup of $\mathrm{Sp}_4(F)$ is denoted by $P(F)$, and the Klingen parabolic subgroup is denoted by $Q(F)$. They have the following Levi decompositions:
\begin{equation*}
P=M_P\ltimes N_P, M_P =\left\{ \left(\begin{smallmatrix} m & \\ & J_2 {}^t m^{-1} J_2\end{smallmatrix}\right):    m\in \mathrm{GL}_2 \right\} , N_P           =\left\{ \left(\begin{smallmatrix} I_2 &Z \\ & I_2\end{smallmatrix}\right):     {}^t Z  J_2 = J_2 Z \right\},
\end{equation*}
\begin{equation*}
Q=M_Q\ltimes N_Q, M_Q   =\left\{ \left(\begin{smallmatrix} a & & \\ & g &\\ &&a^{-1} \end{smallmatrix}\right):    a\in \mathrm{GL}_1, g\in \mathrm{SL}_2 \right\}, N_Q           =\left\{ u(x, y,z)=\left(\begin{smallmatrix} 1 & x & y & z\\ &1 & &y\\ &&1&-x\\&&&1 \end{smallmatrix} \right)\right\}  . 
\end{equation*}

The group $\mathrm{Sp}_2(F)=\mathrm{SL}_2(F)$ embeds into $\mathrm{Sp}_4(F)$ via
\begin{equation*}
t(g):=\left(  \begin{smallmatrix} 1 & &\\ &g &\\&&1 \end{smallmatrix}\right)\in \mathrm{Sp}_4(F), \ \ g\in \mathrm{SL}_2(F).
\end{equation*} 
We regard elements of $\Sp_{2n}(F)$ as elements of $\widetilde{\mathrm{Sp}}_{2n}(F)$ using the trivial section $g\mapsto (g, 1)$ (the map $g\mapsto (g, 1)$ is not a group homomorphism). If $\alpha$ and $\beta$ are genuine functions on $\widetilde{\mathrm{SL}}_{2}(F)$ and $\widetilde{\mathrm{Sp}}_{4}(F)$ respectively, the function $g\mapsto \alpha(g)\beta(t(g))$ is well defined on $\SL_{2}(F)$.

We have the following subgroups of $\mathrm{SL}_2(F)$:
\begin{equation*}
B_{\mathrm{SL}_2}(F)= \left\{ \left(\begin{smallmatrix} a & b \\ & a^{-1}\end{smallmatrix}\right): a\in F^\times, b\in F \right\}, \ \ 
A_2(F)=\left\{ \left(\begin{smallmatrix} a & \\ & a^{-1}\end{smallmatrix}\right):a\in F^\times \right\}, \ \  N_2(F)=\left\{ \left(\begin{smallmatrix} 1 & b\\ & 1\end{smallmatrix}\right):b\in F \right\} .
\end{equation*}

\subsection{The Weil Representation of $\widetilde{\mathrm{SL}}_2(F)\ltimes\mathcal{H}(F)$}
\label{subsection-Weil-representation}
Let $\mathcal{H}(F)$ be the Heisenberg group in three variables, where the multiplication is given by
$$
(x_1, y_1, z_1) (x_2, y_2, z_2)= (x_1+x_2, y_1+y_2, z_1+z_2+x_1y_2-x_2y_1).
$$
The group $\mathrm{SL}_2(F)$ acts on $\mathcal{H}(F)$ via
$$
(x, y, z) \cdot g =((x, y)g, z), \ \ \ g\in \mathrm{SL}_2(F),
$$
where $(x, y)g$ is the usual matrix multiplication. We then form the semi-direct product $\mathrm{SL}_2(F)\ltimes \mathcal{H}(F)$. Associated to the character $\psi$, there is a Weil representation $\omega_\psi$ of the group $ \widetilde{\mathrm{SL}}_2(F) \ltimes \mathcal{H}(F)$, and it can be realized on the Schwartz space $\mathcal{S}(F)$. For $a\in F^\times$, $b\in F$, $(x, y, z)\in \mathcal{H}(F)$ and $\Phi\in \mathcal{S}(F)$, we have the following formulas (see \cite[Sections 2.2 and 2.5]{BerndtSchmidt1998})
\begin{equation*}
\begin{split}
 \omega_\psi  \left( \left(\begin{smallmatrix} a &\\&a^{-1} \end{smallmatrix}\right), \varepsilon \right) \Phi (\xi)       &=  \varepsilon \gamma_\psi(a) |a|^{{1/2}}   \Phi(\xi a),  \\
  \omega_\psi \left( \left(\begin{smallmatrix} 1 & b\\&1 \end{smallmatrix}\right) , \varepsilon\right) \Phi (\xi)       &= \varepsilon \psi(b\xi^2)\Phi(\xi),\\
  \omega_\psi(((x, y, z), \varepsilon) )\Phi(\xi)      & =  \varepsilon\psi(z+2\xi y+xy )\Phi(\xi+x).
\end{split}
\end{equation*}
Here, $\gamma_\psi(a)$ is the inverse of the Weil factor associated with the character $\psi$, which satisfies the following properties (see \cite[Appendix]{Rao1993}):
\begin{equation*}
    \gamma_\psi(a b)=\gamma_\psi(a)\cdot \gamma_\psi(b) \cdot (a,b)_F,\, 
    \gamma_\psi(b^2)=1,\,  \gamma_\psi(ab^2)=\gamma_\psi(a),\,  \gamma_\psi(a)^4=1\, \text{ for all } a, b\in F^\times.
\end{equation*}

The unipotent group $N_Q(F)$ is isomorphic to the Heisenberg group $\mathcal{H}(F)$ via the following map
\begin{equation} 
\label{eq-NQ-Heisenberg}
\begin{split}
N_Q(F)    &\to \mathcal{H}(F)           \\
 \left(\begin{smallmatrix} 1 & x & y & z\\ &1 & &y\\ &&1&-x\\&&&1 \end{smallmatrix}\right)  &\mapsto  (x, y, z).
\end{split}
\end{equation}

\begin{remark}
\label{remark-1}
We remark that the above isomorphism is different from the isomorphism in \cite[Section 1.3]{GinzburgRallisSoudry1998}, by a factor of $2$ on the $y$-coordinate. The isomorphism in \eqref{eq-NQ-Heisenberg} is more natural, while the isomorphism in \cite[Section 1.3]{GinzburgRallisSoudry1998} makes the unipotent matrix $u(0,y,0)\in N_Q(F)$ act by $\psi(y)$ under the Weil representation. We remark that both the isomorphism in \eqref{eq-NQ-Heisenberg} and the isomorphism in \cite[Section 1.3]{GinzburgRallisSoudry1998} are commonly used (see, for example, \cite{GinzburgSoudry2020, Yan2021, Zhang2018}). 
\end{remark}

\subsection{The tensor product $L$-function for $\mathrm{SL}_2\times \mathrm{GL}_2$}
In this subsection we define the $L$-functions that we study in this paper. 

Let $\pi$ be an irreducible unramified representation of $\mathrm{SL}_2(F)$ and $\tau$ be an irreducible unramified representation of $\mathrm{GL}_2(F)$. Note that the $L$-group of $\mathrm{SL}_2$ is $\mathrm{SO}_3(\mathbb{C})$. Let
\begin{equation*}
\begin{split}
t_{\pi}=\mathrm{diag}(b_1, 1, b_1^{-1})
\end{split}
\end{equation*}
be the semisimple conjugacy class in $\mathrm{SO}_3(\mathbb{C})$ attached to $\pi$. Similarly, let 
\begin{equation*}
\begin{split}
t_{\tau}=\mathrm{diag}(a_{1}, a_{2})
\end{split}
\end{equation*}
be the semisimple conjugacy class in $\mathrm{GL}_2(\mathbb{C})$ attached to $\tau$. Then the local tensor product $L$-function for $\pi \times \tau$ is defined by
$$
L(\pi \times \tau, s)=\det(1 - t_{\pi} \otimes t_{\tau} q^{-s})^{-1}.
$$
The local symmetric square $L$-function for $\tau$ is
\begin{equation*}
L(\tau, \mathrm{Sym}^2, s)=(1-a_{1}^2 q^{-s})^{-1} (1-a_{1} a_{2} q^{-s})^{-1} (1-a_{2}^2 q^{-s})^{-1} .
\end{equation*}

\section{A new proof of the unramified computation}
\label{section-unramified}

In this section we compute the local unramified integral  corresponding to the global integral \eqref{eq-global-integral-prop-afterunfolding} after unfolding.  

Let $\pi$ be an irreducible unramified $\psi^{-1}$-generic summand of the induced representation
$$
\mathrm{Ind}_{B_{\mathrm{SL}_2}(F)}^{\mathrm{SL}_2(F)}(\chi) \ \ \text{ (normalized induction)}
$$
where $\chi$ is an unramified quasi-characters of $F^\times$. Let $\tau$ be an irreducible unramified generic principal series representation of $\mathrm{GL}_2(F)$, with 
$$\tau=\mathrm{Ind}_{B_{\mathrm{GL}_2}(F)}^{\mathrm{GL}_2(F)}(\chi_1 \otimes \chi_2) \ \   \text{(normalized induction)}
$$ 
where $\chi_1, \chi_2$ are unramified quasi-characters of $F^\times$. Note that $2$ is used as a subscript for the multiplicative character $\chi_2$, but $\psi_2$ is the additive character given by $x\mapsto \psi(2x)$, and we hope that the use of 2 is clear from the context.

For $a\in F^\times$, recall that $\gamma_\psi(a)$ was defined in Section~\ref{subsection-Weil-representation}. We also regard $\gamma_\psi$ as a function on $\GL_2(F)$ via the pullback of $\gamma_\psi$ by the determinant map.
Consider the space
$$
  \mathrm{Ind}_{\widetilde{P}(F)}^{\widetilde{\mathrm{Sp}}_4(F)}(\mathcal{W}(\tau, \psi_2) \otimes   |\det|^s \otimes \gamma_\psi^{-1} ) \ \ \text{ (normalized induction)},
$$
consisting of smooth functions $\tilde{f}_{\mathcal{W}(\tau,\psi_2), s} $ on the group $\widetilde{\mathrm{Sp}}_4(F)$ which takes values in $\mathcal{W}(\tau, \psi_2)$, i.e., for any $g\in \widetilde{\mathrm{Sp}}_4(F)$, there is a function $W^g_{\tau, s}\in \mathcal{W}(\tau, \psi_2)$ (depending on $g$) such that 
\begin{equation}
\begin{split}
\tilde{f}_{\mathcal{W}(\tau,\psi_2), s}  \left(   \left( \left(\begin{smallmatrix} m & Z\\ & m^* \end{smallmatrix}\right) , \varepsilon\right)  g   \right) = \varepsilon \gamma_{\psi}(\det(m))^{-1} |\det(m)|^{s+\frac{3}{2}} W_{\tau, s}^g(m)
\end{split}
\label{eq-local-induction-space1}
\end{equation}
where $m\in \mathrm{GL}_2(F)$, $\varepsilon\in \{\pm 1\}$. 

The local unramified integral corresponding to the global integral in \eqref{eq-global-integral-prop-afterunfolding} is 
\begin{equation*}
I(W_\pi^0, \Phi^0, \tilde{f}^0_{\mathcal{W}(\tau,\psi_2), s} ) : =  \int\limits_{N_2(F)\backslash \mathrm{SL}_2(F)} W^0_\pi(g)  \int\limits_{N_Q^0(F)} \left(\omega_\psi(rg)\Phi^0\right)(1) \tilde{f}^0_{\mathcal{W}(\tau,\psi_2), s}  \left( \gamma r t(g)  \right) dr dg.
\end{equation*}
Here $W^0_\pi\in \mathcal{W}(\pi, \psi^{-1})$ is the unramified Whittaker function for $\pi$ normalized so that $W_\pi^0(I_2)=1$,  $\Phi^0=\text{1}_{\mathcal{O}_F}$ is the characteristic function of $\mathcal{O}_F$, and 
$\tilde{f}_{\mathcal{W}(\tau,\psi_2), s} ^0 \in  \mathrm{Ind}_{\widetilde{P}(F)}^{\widetilde{\mathrm{Sp}}_4(F)}(\mathcal{W}(\tau, \psi_2) \otimes   |\det|^s \otimes \gamma_\psi^{-1} )$ is the normalized unramified section so that its value at the identity is exactly the normalized unramified Whittaker function $W_\tau^0$ in $\mathcal{W}(\tau, \psi_2)$ which has value 1 at the identity. Note that for $r\in N_Q^0(F)$, the function $g\mapsto \left(\omega_\psi(rg)\Phi^0\right)(1) \tilde{f}^0_{\mathcal{W}(\tau,\psi_2), s}  \left( \gamma r t(g)  \right)$ is well-defined on $\SL_2(F)$. Also, the map $(h, \varepsilon)\mapsto (t(h), \varepsilon)$ is an embedding of $\widetilde{\mathrm{Sp}}_2(F)$ in $\widetilde{\mathrm{Sp}}_4(F)$. To see this, one can take $(g, \varepsilon)=(I_2, 1)\in \widetilde{\mathrm{Sp}}_2(F)$ in the homomorphism given by \cite[(1.28)]{GinzburgSoudry2021} and note that Ranga Rao's $x$-function is trivial on the identity and that $(x(I_2), x(h))_F=1$. Moreover, the integral $I(W_\pi^0, \Phi^0, \tilde{f}^0_{\mathcal{W}(\tau,\psi_2), s} )$ converges absolutely for $\mathrm{Re}(s)\gg 0$  \cite[Proposition 6.5]{GinzburgRallisSoudry1998}. We re-state Theorem~\ref{thm-Introduction}~\ref{thm-Introduction-part(ii)} as follows.

\begin{theorem}\cite[Theorem 6.3]{GinzburgRallisSoudry1998}
For $\mathrm{Re}(s)\gg 0$, we have
\begin{equation}
I(W_\pi^0, \Phi^0, \tilde{f}^0_{\mathcal{W}(\tau,\psi_2), s} ) = \frac{L(\pi\times \tau, s+\frac{1}{2}) }{L(\tau, \mathrm{Sym}^2, 2s+1)} .
\label{eq-local-unramified-main-result}
\end{equation}
\label{thm-local-unramified-main-result}
\end{theorem}

The rest of this section is devoted to a proof of Theorem \ref{thm-local-unramified-main-result}, which is different from \cite{GinzburgRallisSoudry1998}. Our method is based on the Casselman-Shalika formulas for $\SL_2$ and $\GL_2$ which we recall below.

\begin{lemma}\cite[Theorem 5.4]{CasselmanShalika1980II}
Let $a\in F^\times$. Then 
\begin{equation*}
W_\pi^0  \left( \begin{smallmatrix} a & \\ & a^{-1}\end{smallmatrix}\right) =  
 \begin{cases}
|a| \cdot \frac{\chi(\varpi)^{\ord(a)+1}-\chi(\varpi)^{-\ord(a)}}{\chi(\varpi)-1}   &\text{ if } \ord(a)\ge 0 \\
0 & \text{ if }\ord(a)< 0
\end{cases}
\end{equation*}
and 
\begin{equation*}
W_\tau^0\left( \begin{smallmatrix} a & \\&1\end{smallmatrix}\right)=
 \begin{cases}
 |a|^{\frac{1}{2}} \cdot   \frac{\chi_1(\varpi)^{\ord(a)+1}-\chi_2(\varpi)^{\ord(a)+1} }{\chi_1(\varpi)-\chi_2(\varpi)} &\text{ if }\ord(a)\ge 0\\
 0 &\text{ if }\ord(a)<0
 \end{cases}.
 \end{equation*}  
 \label{lemma-CasselmanShalika}
\end{lemma}

Using the Iwasawa decomposition $\mathrm{SL}_2(F)=N_2(F)A_2(F)\mathrm{SL}_2(\mathcal{O}_F)$, we have
\begin{equation*}
\begin{split}
&I(W_\pi^0, \Phi^0, \tilde{f}^0_{\mathcal{W}(\tau,\psi_2), s} )\\
 =         &\int\limits_{N_2(F)\backslash \mathrm{SL}_2(F)} W_\pi^0(g) \int\limits_{N_Q^0(F)} \left(\omega_\psi(rg  )\Phi^0\right)(1) \tilde{f}_{\mathcal{W}(\tau,\psi_2), s} ^0 \left( \gamma r t(g)  \right) dr dg   \\
=     & \int\limits_{F^\times} W_\pi^0 \left(\begin{smallmatrix} a&\\ & a^{-1}\end{smallmatrix}\right) \int\limits_{F^2} \left(\omega_\psi \left( (x, 0, z)     \left(\begin{smallmatrix} a& \\&a^{-1}\end{smallmatrix}  \right)\right)\Phi^0\right)(1)     \tilde{f}_{\mathcal{W}(\tau,\psi_2), s} ^0 \left(\gamma u(x, 0, z) t(a)  \right)  dxdz  |a|^{-2} d^\times a  \\
=      &    \int\limits_{F^\times}  W_\pi^0 \left(\begin{smallmatrix} a&\\ & a^{-1}\end{smallmatrix}\right) \int\limits_{F^2} \left( \omega_\psi( \left(\begin{smallmatrix} a& \\&a^{-1}\end{smallmatrix} \right)  (xa, 0, z)     )\Phi^0\right)(1)     \tilde{f}_{\mathcal{W}(\tau,\psi_2), s} ^0 \left(\gamma t(a) u(xa, 0, z) \right)dxdz   |a|^{-2}  d^\times a.
\end{split}
\end{equation*}
We remind the reader that we have identified an element $r=u(x, 0, z)\in N_Q^0(F)$ with an element $(x, 0,z)$ in the Heisenberg group $\mathcal{H}(F)$.
By a change of variable $x\mapsto xa^{-1}$ and by Lemma \ref{lemma-CasselmanShalika}, we get
$$
 \int\limits_{F^\times\cap \mathcal{O}_F}  W_\pi^0 \left(\begin{smallmatrix} a&\\ & a^{-1}\end{smallmatrix}\right)       \int\limits_{F^2}  \left(\omega_\psi(  \left(\begin{smallmatrix} a& \\&a^{-1}\end{smallmatrix} \right)  (x, 0, z)  )\Phi^0\right)(1)     \tilde{f}_{\mathcal{W}(\tau,\psi_2), s} ^0 \left(\gamma t(a) u(x, 0, z)   \right)dxdz   |a|^{-3}  d^\times a.
$$
Note that 
$$
\left(\omega_\psi( \left( \begin{smallmatrix} a& \\&a^{-1}\end{smallmatrix} \right)  (x, 0, z)    )\Phi^0\right)(1)  =  
 |a|^{1/2} \gamma_{\psi}(a) \left( \omega_\psi( (x, 0, z)  )   \Phi^0 \right)(a).
$$
Conjugating $t(a)$ to the left of $\gamma$, then we have
\begin{equation*}
\begin{split}
\int\limits_{F^\times \cap \mathcal{O}_F}  W_\pi^0 \left(\begin{smallmatrix} a&\\ & a^{-1}\end{smallmatrix}\right) \int\limits_{F^2}  \left( \omega_\psi( (x, 0, z))  \Phi^0\right)(a)       \tilde{f}_{\mathcal{W}(\tau,\psi_2), s} ^0 \left( \left(\begin{smallmatrix} a &&&\\&1&&\\&&1&\\ &&&a^{-1} \end{smallmatrix}\right) \gamma u(x, 0, z) \right)   
  dxdz   \gamma_{\psi}(a)  |a|^{-\frac{5}{2}} d^\times a.
\end{split}
\end{equation*}
Since $\Phi^0$ is supported in $\mathcal{O}_F$, the section $ \tilde{f}_{\mathcal{W}(\tau,\psi_2), s} ^0$ is unramified, and by the formula
$$
\left(\omega_\psi((x,0,z))\Phi^0\right)(a)=\left(\omega_\psi((0,0,z))\Phi^0\right)(a+x)=\psi(z)\Phi^0(a+x) ,
$$
we see that the integration over $x$ evaluates to 1. It's convenient to use Jacquet's style notation for the section $\tilde{f}_{\mathcal{W}(\tau,\psi_2), s} ^0$, which we will use for the remainder of the paper.
Note that
\begin{equation*}
 \tilde{f}_{\mathcal{W}(\tau,\psi_2), s} ^0 \left( I_2; \left(\begin{smallmatrix} a &&&\\&1&&\\&&1&\\ &&&a^{-1} \end{smallmatrix}\right) \gamma u(0, 0, z) \right)  = \gamma_\psi(a)^{-1} |a|^{s+3/2}  \tilde{f}_{\mathcal{W}(\tau,\psi_2), s} ^0 \left(  \left(\begin{smallmatrix} a &\\&1 \end{smallmatrix}\right);    \gamma u(0, 0, z)\right)  
\end{equation*}
since the representation $\mathcal{W}(\tau, \psi_2) \otimes   |\det|^s \otimes \gamma_\psi^{-1}$ acts on the Levi part by 
$$
( \left( \begin{smallmatrix} m & \\ & m^* \end{smallmatrix}\right), 1) \mapsto  \gamma_\psi(\det(m))^{-1} |\det(m)|^{s+3/2} \mathcal{W}(\tau, \psi_2)(m).
$$
Thus
\begin{equation}
\begin{split}
&I(W_\pi^0, \Phi^0, \tilde{f}^0_{\mathcal{W}(\tau,\psi_2), s} )\\
=& \int\limits_{F^\times \cap \mathcal{O}_F}  W_\pi^0 \left(\begin{smallmatrix} a&\\ & a^{-1}\end{smallmatrix}\right) \int\limits_{F}  \left( \omega_\psi( (0, 0, z))  \Phi^0\right)(a)       \tilde{f}_{\mathcal{W}(\tau,\psi_2), s} ^0 \left(  \left(\begin{smallmatrix} a &\\&1  \end{smallmatrix}\right) ; \gamma u(0, 0, z)\right)   
  dz   |a|^{s-1} d^\times a  \\
 =& \sum_{k=0}^\infty   \int\limits_{\varpi^k \mathcal{O}_F^\times}  W_\pi^0 \left(\begin{smallmatrix} a&\\ & a^{-1}\end{smallmatrix}\right) \int\limits_{F}   \left(\omega_\psi( (0, 0, z))  \Phi^0\right)(a)       \tilde{f}_{\mathcal{W}(\tau,\psi_2), s} ^0  \left( \left(\begin{smallmatrix} a &\\&1  \end{smallmatrix}\right) ; \gamma u(0, 0, z) \right)   
  dz      |a|^{s-1} d^\times a \\
  =&  \sum_{k=0}^\infty   \int\limits_{\varpi^k \mathcal{O}_F^\times}  W_\pi^0 \left(\begin{smallmatrix} a&\\ & a^{-1}\end{smallmatrix}\right) J(a)    |a|^{s-1} d^\times a
\end{split}
\label{eq-local-zeta-integral-1}
\end{equation}
where 
$$
J(a) := \int\limits_{F}  \left( \omega_\psi( (0, 0, z))  \Phi^0\right)(a)       \tilde{f}_{\mathcal{W}(\tau,\psi_2), s} ^0  \left( \left(\begin{smallmatrix} a &\\&1  \end{smallmatrix}\right) ; \gamma u(0, 0, z)  \right)   
  dz .
$$
By dividing the domain of $z$ into two parts, we can write
$$
J(a)=J_1(a)+J_2(a)
$$
where
\begin{equation*}
J_1(a) : = \int\limits_{\mathcal{O}_F}   \omega_\psi( (0, 0, z))  \Phi^0(a)       \tilde{f}_{\mathcal{W}(\tau,\psi_2), s} ^0  \left(  \left(\begin{smallmatrix} a &\\&1  \end{smallmatrix}\right) ; \gamma u(0, 0, z)  \right)   
  dz 
\end{equation*}
and 
\begin{equation*}
J_2(a) : =  \int\limits_{F \setminus  \mathcal{O}_F}   \omega_\psi( (0, 0, z))  \Phi^0(a)       \tilde{f}_{\mathcal{W}(\tau,\psi_2), s} ^0  \left(  \left(\begin{smallmatrix} a &\\&1  \end{smallmatrix}\right) ; \gamma u(0, 0, z)  \right)   
  dz .
\end{equation*}

The values of $J_1(a)$ and $J_2(a)$ are computed in the following two propositions. 

\begin{proposition}
For $a\in F^\times \cap \mathcal{O}_F$, we have 
\begin{equation}
J_1(a)=  |a|^{\frac{1}{2}} \cdot \frac{\chi_1(\varpi)^{\ord(a)+1}-\chi_2(\varpi)^{\ord(a)+1}}{\chi_1(\varpi)-\chi_2(\varpi)}.
\label{eq-unramified-J1}
\end{equation}
\label{prop-unramified-J1}
\end{proposition}

\begin{proof}
Since both the section $\tilde{f}_{\mathcal{W}(\tau,\psi_2), s} ^0$ and the character $\psi$ are unramified, we have 
\begin{equation*}
\begin{split}
J_1(a)  &= \int\limits_{\mathcal{O}_F}   \left(\omega_\psi( (0, 0, z))  \Phi^0\right)(a)      \tilde{f}_{\mathcal{W}(\tau,\psi_2), s} ^0  \left( \left(\begin{smallmatrix} a &\\&1  \end{smallmatrix}\right) ; \gamma u(0, 0, z)  \right)   
  dz   \\
  &= \int\limits_{\mathcal{O}_F}  \psi(z)   \tilde{f}_{\mathcal{W}(\tau,\psi_2), s} ^0  \left( \left(\begin{smallmatrix} a &\\&1  \end{smallmatrix}\right) ; I_4 \right)   
  dz \\
  &=W_\tau^0 \left( \begin{smallmatrix} a&\\  &1\end{smallmatrix}\right) \\
  &= |a|^{\frac{1}{2}} \cdot  \frac{\chi_1(\varpi)^{\ord(a)+1}-\chi_2(\varpi)^{\ord(a)+1}}{\chi_1(\varpi)-\chi_2(\varpi)}
  \end{split}
\end{equation*}
where the last equality follows from Lemma \ref{lemma-CasselmanShalika}.
\end{proof}

\begin{proposition}
For $a\in F^\times \cap \mathcal{O}_F$, we have 
\begin{equation}
J_2(a)=   |a|^{\frac{1}{2}} \cdot \frac{\chi_1(\varpi)^{\ord(a)}-\chi_2(\varpi)^{\ord(a)}}{\chi_1(\varpi)-\chi_2(\varpi)}\cdot \chi_1(\varpi)\chi_2(\varpi) \cdot q^{-(s+1/2)}.
\label{eq-unramified-J2}
\end{equation}
\label{prop-unramified-J2}
\end{proposition}

\begin{proof}
To deal with the integral when $z\in F\setminus \mathcal{O}_F$, we consider the following matrix identity
\begin{equation*}
\begin{split}
\gamma  u(0,0,z) \gamma^{-1}   &=   \left( \begin{smallmatrix} 1 &&&\\&1&&\\&-z&1&\\ &&&1 \end{smallmatrix}\right) =     \left(\begin{smallmatrix} 1 &&&\\&z^{-1}&-1&\\&&z&\\ &&&1\end{smallmatrix}\right)    \left(\begin{smallmatrix} 1 &&&\\&&1&\\&-1&&\\ &&&1\end{smallmatrix}\right)   \left(\begin{smallmatrix} 1 &&&\\&1&-z^{-1}&\\&&1&\\ &&&1\end{smallmatrix}\right) .
\end{split}
\end{equation*}
Then
\begin{equation*}
\begin{split}
J_2(a) &= \sum_{m=1}^\infty  \int\limits_{ \varpi^{-m}  \mathcal{O}_F^\times}   \omega_\psi( (0, 0, z))  \Phi^0(a)       \tilde{f}_{\mathcal{W}(\tau,\psi_2), s} ^0  \left( \left(\begin{smallmatrix} a &\\&1  \end{smallmatrix}\right); \gamma u(0, 0, z)  \right)     
  dz     \\
 & = \sum_{m=1}^\infty  \int\limits_{ \varpi^{-m}  \mathcal{O}_F^\times}    \tilde{f}_{\mathcal{W}(\tau,\psi_2), s} ^0 \left( \left(\begin{smallmatrix} a & \\&1 \end{smallmatrix}\right);  \left(\begin{smallmatrix} 1 &&&\\&z^{-1}&-1&\\&&z&\\ &&&1\end{smallmatrix}\right)    \left(\begin{smallmatrix} 1 &&&\\&&1&\\&-1&&\\ &&&1\end{smallmatrix}\right)   \left(\begin{smallmatrix} 1 &&&\\&1&-z^{-1}&\\&&1&\\ &&&1\end{smallmatrix}\right)      \right)   \psi(z)dz \\
 & = \sum_{m=1}^\infty  \int\limits_{ \varpi^{-m}  \mathcal{O}_F^\times}    \tilde{f}_{\mathcal{W}(\tau,\psi_2), s} ^0  \left(\left(\begin{smallmatrix} a & \\&1 \end{smallmatrix}\right)  ; \left(\begin{smallmatrix} 1&&&\\&z^{-1}&&\\&&z&\\ &&&1\end{smallmatrix}\right)    \right)        \psi(z)dz \\
 &=   \sum_{m=1}^\infty  \int\limits_{ \varpi^{-m} \mathcal{O}_F^\times}   \gamma_\psi(z^{-1})^{-1} |z^{-1}|^{s+3/2} W_\tau^0  \left( \begin{smallmatrix} a & \\ & z^{-1} \end{smallmatrix}\right) \psi(z)dz.
\end{split}
\end{equation*}
Note that
\begin{equation*}
\begin{split}
W_\tau^0  \left( \begin{smallmatrix} a & \\ & z^{-1} \end{smallmatrix}\right) &=  W_\tau^0  \left(   \left( \begin{smallmatrix} z^{-1} & \\ & z^{-1} \end{smallmatrix}\right)  \left( \begin{smallmatrix} az & \\ & 1 \end{smallmatrix}\right)    \right) \\
&=\chi_1(z^{-1})\chi_2(z^{-1}) W_\tau^0  \left( \begin{smallmatrix} az & \\ & 1 \end{smallmatrix}\right)\\
  &=  \begin{cases}
\chi_1(z^{-1})\chi_2(z^{-1}) |az|^{\frac{1}{2}}  \cdot \frac{\chi_1(\varpi)^{\ord(az)+1}-\chi_2(\varpi)^{\ord(az)+1}}{\chi_1(\varpi)-\chi_2(\varpi)}       & \text{ if }\ord(z)\ge -\ord(a) \\
0 & \text{ if }\ord(z)<-\ord(a)
\end{cases}
\end{split}
\end{equation*}
where the last equality follows from Lemma \ref{lemma-CasselmanShalika}. If $\ord(a)=0$, then for any $z\in  \varpi^{-m} \mathcal{O}_F^\times$ with $m\ge 1$, we have $\ord(z)< -\ord(a)$, and hence $J_2(a)=0$. If $\ord(a)>0$, then
\begin{equation}
\begin{split}
J_2(a) = \sum_{m=1}^{\ord(a)}  \int\limits_{ \varpi^{-m} \mathcal{O}_F^\times}    \gamma_\psi(z^{-1})^{-1} |z^{-1}|^{s+3/2} W_\tau^0  \left( \begin{smallmatrix} a & \\ & z^{-1} \end{smallmatrix}\right) \psi(z)dz.
\end{split}
\label{eq-unramified-J2-reduction1}
\end{equation}

Now we assume $\ord(a)>0$. Then 
\begin{equation*}
\begin{split}
J_2(a)  &= \sum_{m=1}^{\ord(a)}  \int\limits_{ \varpi^{-m} \mathcal{O}_F^\times}   \gamma_\psi(z^{-1})^{-1} |z^{-1}|^{s+3/2}  \chi_1(z^{-1})\chi_2(z^{-1}) |az|^{\frac{1}{2}}  \cdot \frac{\chi_1(\varpi)^{\ord(az)+1}-\chi_2(\varpi)^{\ord(az)+1}}{\chi_1(\varpi)-\chi_2(\varpi)}      \psi(z)dz \\
&=    |a|^{\frac{1}{2}}  \sum_{m=1}^{\ord(a)}     \chi_1(\varpi^m)\chi_2(\varpi^m)  
\frac{\chi_1(\varpi)^{\ord(a)-m+1}-\chi_2(\varpi)^{\ord(a)-m+1}}{\chi_1(\varpi)-\chi_2(\varpi)}   \cdot q^{-m(s+1)}      \int\limits_{ \varpi^{-m} \mathcal{O}_F^\times} \gamma_\psi(z^{-1})^{-1}  \psi(z)dz \\
&=    |a|^{\frac{1}{2}}  \sum_{m=1}^{\ord(a)}     \chi_1(\varpi^m)\chi_2(\varpi^m)  
\frac{\chi_1(\varpi)^{\ord(a)-m+1}-\chi_2(\varpi)^{\ord(a)-m+1}}{\chi_1(\varpi)-\chi_2(\varpi)}   \cdot q^{-ms}      \int\limits_{\mathcal{O}_F^\times} \gamma_\psi( \varpi^{m} u^{-1})^{-1}  \psi( \varpi^{-m} u)du.
\end{split}
\end{equation*}
We remind the reader that here both $dz$ and $du$ are additively invariant. 

In Lemma \ref{lemma-local-integral-gamma-psi-vanishing} below, we will prove that only the term corresponding to $m=1$ contributes to the sum. Assuming Lemma \ref{lemma-local-integral-gamma-psi-vanishing} for the moment, then for $\ord(a)>0$ we have
\begin{equation*}
J_2(a)= |a|^{\frac{1}{2}} \cdot  \chi_1(\varpi)\chi_2(\varpi)  
\frac{\chi_1(\varpi)^{\ord(a)}-\chi_2(\varpi)^{\ord(a)}}{\chi_1(\varpi)-\chi_2(\varpi)}   \cdot q^{-s}      \int\limits_{\mathcal{O}_F^\times} \gamma_\psi( \varpi u^{-1})^{-1}  \psi( \varpi^{-1} u)du.
\end{equation*}
It follows from the formula $J_1(\mathbb{F},\psi,\chi^0)=q^{-\frac{1}{2}}$ in the proof of \cite[Lemma 1.12]{Szpruch2009} that
\begin{equation*}
\begin{split}
\int\limits_{\mathcal{O}_F^\times} \gamma_\psi( \varpi u^{-1})^{-1}  \psi(\varpi^{-1} u )  du =q^{-1/2} .
\end{split}
\end{equation*}
Therefore, 
\begin{equation*}
\begin{split}
J_2(a) &=   |a|^{1/2}\cdot   \frac{\chi_1(\varpi)^{\ord(a)}  - \chi_2(\varpi)^{\ord(a)} }{\chi_1(\varpi)-\chi_2(\varpi)} \cdot  \chi_1(\varpi)\chi_2(\varpi)  \cdot  q^{-s-1/2} .
\end{split}
\end{equation*}
Note that the above formula also applies to $J_2(a)=0$ when $\ord(a)=0$.

Thus, we will finish the proof of Proposition \ref{prop-unramified-J2} once we take care of Lemma \ref{lemma-local-integral-gamma-psi-vanishing}.
\end{proof}

\begin{lemma}
For any $m\ge 2$, we have
$$
  \int\limits_{\mathcal{O}_F^\times} \gamma_\psi( \varpi^{m} u^{-1})^{-1}  \psi( \varpi^{-m} u)du=0.
$$
\label{lemma-local-integral-gamma-psi-vanishing}
\end{lemma}

\begin{proof}
This follows from the proof of \cite[Lemma 1.11]{Szpruch2009}. 

For even $m$, we take $\chi=1$ in the proof of \cite[Lemma 1.11]{Szpruch2009}, and note that $e=0$ since the residue characteristic of $F$ is odd. 

Now we consider the case where $m$ is odd (hence $m\ge 3$). 
Note that
\begin{equation*}
    \gamma_\psi( \varpi^{m} u^{-1})^{-1} = \gamma_\psi( \varpi^{m} u^{-1} (\varpi^{-m} u)^2 )^{-1} =\gamma_\psi( \varpi^{-m} u)^{-1}.
\end{equation*}
Thus
\begin{equation}
\label{eq-local-integral-with-gamma-psi}
\int\limits_{\mathcal{O}_F^\times} \gamma_\psi( \varpi^{m} u^{-1})^{-1}  \psi( \varpi^{-m} u)du=\int\limits_{\mathcal{O}_F^\times} \gamma_\psi( \varpi^{-m} u )^{-1}  \psi( \varpi^{-m} u)du.
\end{equation}
The right-hand side of \eqref{eq-local-integral-with-gamma-psi} is equal to the integral $J_m(\mathbb{F}, \psi, \chi)$ in \cite{Szpruch2009} with $\chi=1$, which is zero by the proof of \cite[Lemma 1.11]{Szpruch2009}. This finishes the proof of Lemma \ref{lemma-local-integral-gamma-psi-vanishing}.
\end{proof}

Now we plug in the formulas for $J_1(a)$ and $J_2(a)$ from Proposition \ref{prop-unramified-J1} and Proposition \ref{prop-unramified-J2} and the formula for $W_\pi^0$ from Lemma \ref{lemma-CasselmanShalika} into (\ref{eq-local-zeta-integral-1}), so we get
\begin{equation}
\begin{split}
&I(W_\pi^0, \Phi^0, \tilde{f}^0_{\mathcal{W}(\tau,\psi_2), s} )\\
  =&  \sum_{k=0}^\infty   \int\limits_{\varpi^k \mathcal{O}_F^\times}   \frac{\chi(\varpi)^{\ord(a)+1}-\chi(\varpi)^{-\ord(a)}}{\chi(\varpi)-1}  \cdot \frac{\chi_1(\varpi)^{\ord(a)+1}-\chi_2(\varpi)^{\ord(a)+1}}{\chi_1(\varpi)-\chi_2(\varpi)}    |a|^{s+\frac{1}{2}} d^\times a \\
  &  + \sum_{k=1}^\infty   \int\limits_{\varpi^k \mathcal{O}_F^\times}   \cdot \frac{\chi(\varpi)^{\ord(a)+1}-\chi(\varpi)^{-\ord(a)}}{\chi(\varpi)-1}  \cdot  \frac{\chi_1(\varpi)^{\ord(a)}-\chi_2(\varpi)^{\ord(a)}}{\chi_1(\varpi)-\chi_2(\varpi)}\cdot \chi_1(\varpi)\chi_2(\varpi) \cdot q^{-(s+\frac{1}{2})}   |a|^{s+\frac{1}{2}} d^\times a \\
  =&   \sum_{k=0}^\infty   \frac{\chi(\varpi)^{k+1}-\chi(\varpi)^{-k}}{\chi(\varpi)-1}  \cdot \frac{\chi_1(\varpi)^{k+1}-\chi_2(\varpi)^{k+1}}{\chi_1(\varpi)-\chi_2(\varpi)}    q^{-k(s+\frac{1}{2})} \\
  & +  \sum_{k=1}^\infty     \frac{\chi(\varpi)^{k+1}-\chi(\varpi)^{-k}}{\chi(\varpi)-1}  \cdot  \frac{\chi_1(\varpi)^{k}-\chi_2(\varpi)^{k}}{\chi_1(\varpi)-\chi_2(\varpi)}\cdot \chi_1(\varpi)\chi_2(\varpi) \cdot q^{-(k+1)(s+\frac{1}{2})}.  
\end{split}
\label{eq-local-zeta-integral2}
\end{equation}
Note that
\begin{equation*}
 \frac{\chi_1(\varpi)^{k+1}-\chi_2(\varpi)^{k+1}}{\chi_1(\varpi)-\chi_2(\varpi)} =p_k(\chi_1(\varpi), \chi_2(\varpi))
\end{equation*}
where $p_k(x_1, x_2)$ is the complete homogeneous symmetric polynomial of degree $k$ in two variables, whose generating function (see \cite[(2.5)]{Macdonald1995}) is
\begin{equation*}
\sum_{k=0}^\infty p_k(x_1, x_2)  t^k=  \frac{1}{(1-x_1 t)(1-x_2 t)}.
\end{equation*}
To ease notation, we denote
\begin{equation*}
A:=\chi(\varpi),  \ \ a_1:=\chi_1(\varpi),\ \  a_2:=\chi_2(\varpi),\ \  X:=q^{-(s+\frac{1}{2})}.
\end{equation*}
Then 
\begin{equation*}
\begin{split}
L(\pi\times \tau, s+\frac{1}{2}) &= \frac{1 }{(1-Aa_1 X)(1-Aa_2 X)(1-a_1 X)(1-a_2X)(1-A^{-1}a_1 X)(1-A^{-1}a_2 X)} , \\
L(\tau, \mathrm{Sym}^2, 2s+1) &= \frac{1}{(1-a_1^2 X^2 )(1-a_1 a_2 X^2)(1-a_2^2 X^2)}.
\end{split}
\end{equation*}
The first summation in (\ref{eq-local-zeta-integral2}) can be computed as follows:
\begin{equation*}
\begin{split}
&\sum_{k=0}^\infty   \frac{\chi(\varpi)^{k+1}-\chi(\varpi)^{-k}}{\chi(\varpi)-1}  \cdot \frac{\chi_1(\varpi)^{k+1}-\chi_2(\varpi)^{k+1}}{\chi_1(\varpi)-\chi_2(\varpi)}    q^{-k(s+\frac{1}{2})} \\
=&\sum_{k=0}^\infty   \frac{A}{A-1}  \cdot  p_k(a_1, a_2)  \cdot A^k X^k- \sum_{k=0}^\infty   \frac{1}{A-1}  \cdot  p_k(a_1, a_2) \cdot A^{-k}  X^k \\
=&   \frac{A}{(A-1)(1-Aa_1 X)(1-Aa_2 X)} - \frac{1}{(A-1)(1-A^{-1}a_1 X)(1-A^{-1}a_2 X)} \\
=&\frac{A(1-A^{-1}a_1 X)(1-A^{-1}a_2 X)-  (1-Aa_1 X)(1-Aa_2 X)    }{(A-1)(1-Aa_1 X)(1-Aa_2 X)(1-A^{-1}a_1 X)(1-A^{-1}a_2 X)} \\
=&   \frac{ A-(a_1+a_2)X+A^{-1} a_1 a_2 X^2-1+A(a_1+a_2)X-A^2 a_1 a_2 X^2   }{(A-1)(1-Aa_1 X)(1-Aa_2 X)(1-A^{-1}a_1 X)(1-A^{-1}a_2 X)} \\
=& \frac{1+(a_1+a_2)X -  \frac{A^2+A+1}{A} a_1 a_2 X^2}{(1-Aa_1 X)(1-Aa_2 X)(1-A^{-1}a_1 X)(1-A^{-1}a_2 X)}.
\end{split}
\end{equation*}
The second summation in (\ref{eq-local-zeta-integral2}) is equal to 
\begin{equation*}
\begin{split}
& \sum_{k=0}^\infty     \frac{\chi(\varpi)^{k+2}-\chi(\varpi)^{-(k+1)}}{\chi(\varpi)-1}  \cdot  \frac{\chi_1(\varpi)^{k+1}-\chi_2(\varpi)^{k+1}}{\chi_1(\varpi)-\chi_2(\varpi)}\cdot \chi_1(\varpi)\chi_2(\varpi) \cdot q^{-(k+2)(s+\frac{1}{2})} \\
=& \sum_{k=0}^\infty \frac{a_1 a_2  A^2 X^{2}}{A-1} \cdot p_k(a_1, a_2) \cdot A^k X^{k} - \sum_{k=0}^\infty \frac{a_1 a_2  A^{-1} X^2}{A-1} p_k(a_1, a_2) \cdot A^{-k} X^{k} \\
=&  \frac{a_1 a_2  A^2 X^{2}}{(A-1) (1-Aa_1 X)(1-Aa_2 X) } -\frac{a_1 a_2 X^2}{A(A-1) (1-A^{-1}a_1 X)(1-A^{-1}a_2 X)  } \\
=&  \frac{ a_1 a_2  A^3 X^{2} (1-A^{-1}a_1 X)(1-A^{-1}a_2 X)   - a_1 a_2X^2(1-Aa_1 X)(1-Aa_2 X)  }{A(A-1) (1-Aa_1 X)(1-Aa_2 X)(1-A^{-1}a_1 X)(1-A^{-1}a_2 X)} \\
=& \frac{1 }{(1-Aa_1 X)(1-Aa_2 X)(1-A^{-1}a_1 X)(1-A^{-1}a_2 X)} \cdot  \left(  \frac{A^2+A+1}{A} a_1 a_2 X^2  -a_1 a_2(a_1+a_2)X^3 -a_1^2 a_2^2 X^4 \right).
 \end{split}
\end{equation*}
Thus, 
\begin{equation*}
\begin{split}
&I(W_\pi^0, \Phi^0, \tilde{f}^0_{\mathcal{W}(\tau,\psi_2), s} ) \\
&=  \frac{1 }{(1-Aa_1 X)(1-Aa_2 X)(1-A^{-1}a_1 X)(1-A^{-1}a_2 X)} \cdot  \left( 1+(a_1+a_2)X   -a_1 a_2(a_1+a_2)X^3 -a_1^2 a_2^2 X^4   \right)\\
&=  \frac{1 }{(1-Aa_1 X)(1-Aa_2 X)(1-A^{-1}a_1 X)(1-A^{-1}a_2 X)}  \cdot  \left( 1-a_1a_2 X^2  \right) \left( 1+a_1 X \right) \left( 1+a_2 X \right)\\
&=\frac{(1-a_1a_2 X^2 )(1-a_1^2 X^2)(1-a_2^2 X^2) }{(1-Aa_1 X)(1-Aa_2 X)(1-A^{-1}a_1 X)(1-A^{-1}a_2 X) (1-a_1 X)(1-a_2 X)}\\
&=\frac{L(\pi\times \tau, s+\frac{1}{2}) }{L(\tau, \mathrm{Sym}^2, 2s+1)}.
\end{split}
\end{equation*}
This completes the proof of Theorem \ref{thm-local-unramified-main-result}.

\section*{Acknowledgements}

I would like to thank my advisor Jim Cogdell for his support and his comments on an earlier version of this paper. I also thank the anonymous referee for helpful comments and suggestions.

\end{document}